\documentclass[11pt]{amsart}
\usepackage{amsrefs}
\usepackage{amsfonts, amsmath, amssymb, amscd, amsthm, bm}
\usepackage{cases}
\usepackage{enumerate}

\usepackage{multicol}
\usepackage{comment}
\usepackage[colorlinks,
linkcolor=blue,
anchorcolor=black,
citecolor=red
]{hyperref}

\newtheorem{thm}{Theorem}[section]
\newtheorem{cor}[thm]{Corollary}

\newtheorem{lem}[thm]{Lemma}
\newtheorem{prop}[thm]{Proposition}
\theoremstyle{definition}
\newtheorem{defn}[thm]{Definition}
\theoremstyle{theorem}
\newtheorem{rem}[thm]{Remark}
\newtheorem{ex}[thm]{Example}

\theoremstyle{claim}
\numberwithin{equation}{section}
\def\besec{K_{-}^{\Phi}}
\def\beric{\Ric_{-}^{\Phi}}

\DeclareMathOperator\tr{tr}
\DeclareMathOperator\Ric{Ric}
\DeclareMathOperator\sgn{sgn}

\begin{document}
\title{Rigidity of $k$-extremal submanifolds in a sphere}

\author{Hang Chen}
\address[Hang Chen]{School of Mathematics and Statistics, Northwestern Polytechnical University, Xi' an 710129, P. R. China \\ email: chenhang86@nwpu.edu.cn}
\thanks{Chen was supported by Shaanxi Fundamental Science Research Project for Mathematics and Physics (Grant No.~22JSQ005)}
\author{Yaru Wang}
\address[Yaru Wang]{School of Mathematics and Statistics, Northwestern Polytechnical University, Xi' an 710129, P. R. China \\ email: wyr142857@mail.nwpu.edu.cn}
\thanks{}

\begin{abstract}
In this paper, we study the rigidity of $k(\ge 1)$-extremal submanifolds in a sphere and prove various pinching theorems under different curvature conditions, including sectional and Ricci curvatures in pointwise and integral sense.
\end{abstract}


\maketitle

\section{Introduction}
The study of rigidity phenomenon for submanifolds under certain curvature pinching conditions is one of important topics in geometry of submanifolds, and the history can date back to 1968, when Simons \cite{Sim68} studied the minimal submanifolds in a sphere.
By using Simons' equation, Simons and Chern-do Carmo-Kobayashi proved the rigidity theorem.
\begin{thm}[\cites{Sim68, CdCK70}]\label{thm-Sim}
	Let $M$ be an $n$-dimensional closed minimal submanifold in a unit sphere $\mathbb {S}^{n+p}$ of dimensional $(n+p)$.
	If the squared norm $S$ of the second fundamental form of $M$ satisfies
	\begin{equation}\label{Simons-pinch-cond}
		S\le \frac{n}{2-1/p},
	\end{equation}
	then either $S\equiv 0$ and $M$ is totally geodesic, or $S\equiv \frac{n}{2-1/p}$. In the latter case, $M$ is either one of the Clifford tori with $p=1$, or the Veronese surface in $\mathbb{S}^4$ with $n=p=2$.
\end{thm}

The Clifford tori are
\begin{equation}
	C_{m,n-m}:=\mathbb {S}^{m}(\sqrt{\frac{m}{n}})\times \mathbb {S}^{n-m}(\sqrt{\frac{n-m}{n}}),  \quad 1\le m \le n-1.
\end{equation}
The standard embedding of $C_{m,n-m}\to \mathbb{S}^{n+1}$ is minimal.

After the work of Simons and Chern-do Carmo-Kobayashi, a lot of rigidity results were proven by different geometers, for instance, \cites{She89, LL92} for scalar curvature pinching, \cites{Yau74,Yau75,Ito75,GX12} for sectional curvature pinching, \cites{Eji79,She92, Li93} for Ricci curvatures pinching, \cite{Xu94} for global rigidity theorem involving the $L_{n/2}$-norm of $S$.
Recently, Wei and the first author proved the following rigidity results under the integral Ricci curvature.
\begin{thm}{\cite{CW21}*{Theorem 1.4}}\label{thm-CW}
	Let $M$ be an $n(\ge 4)$-dimensional closed minimal submanifold in $\mathbb {S}^{n+p}$. Given a real number $\lambda$ satisfying $(n-2)<(n-1)\lambda\le (n-1)$.
	Then there exists an explicit constant $\epsilon(n)$depending only on $n$ such that if
	\begin{equation}\label{CW-Ricci-int-pinch-cond}
		\|\Ric_{-}^{\lambda}\|_{n/2}<\epsilon(n),
	\end{equation}
	then $M$ is totally geodesic.
\end{thm}
The notation $\Ric_{-}^{\lambda}$ in above theorem is defined by $\Ric_{-}^{\lambda}(x)=\max\{0, (n-1)\lambda-\Ric_{\min}(x)\}$, where $\Ric_{\min}(x)$ denotes the smallest eigenvalue of the Ricci tensor at $x\in M$.
Intuitively, the $L_q$-norm $\|\Ric_{-}^{\lambda}\|_{q}=\Big(\int_M \big(\Ric_{-}^{\lambda}\big)^q\Big)^{1/q}$ measures the amount of Ricci curvature lying below the given bound $(n-1)\lambda$.
More generally, we can replace $\lambda$ by a function $\Phi$ on $M$ and define
\begin{equation*}
	\beric(x)=\max\{0,(n-1)\Phi-\Ric_{\min}(x)\}.
\end{equation*}
Similarly, for sectional curvature, we can define
\begin{equation*}
	\besec(x)=\max\{0,\Phi(x)-K_{\min}(x)\},
\end{equation*}
where $K_{\min}(x)$ denotes the smallest of sectional curvature at $x\in M$.
Zhu proved an analogue of Theorem \ref{thm-CW}  under the integral sectional curvature in her master's thesis \cite{Zhu19}.
Very recently, Pan-Xu-Zhao \cite{PXZ23} proved a rigidity theorem for minimal submanifolds in a sphere under the integral scalar curvature.
Most of the results have been generalized from minimal submanifolds in a sphere to submanifolds in a space form with parallel mean curvature, see \cites{Xu94,GX12,XG13,CW21, Zhu19, PXZ23}.

On the other hand, since minimal submanifolds are the critical points of the area functional, this motivates one can study the rigidity for submanifolds which are the critical points of some other functional.
Given $k\ge 1$, for an $n$-dimensional closed submanifold $M$ in $\mathbb{S}^{n+p}$, we define the functional
\begin{equation}
	F_k=\int_M \rho^{2k},
\end{equation}
where $\rho^2=S-nH^2$ and $H$ is the norm of the mean curvature of $M$.
It is well known that $\rho^2$ is non-negative, and $\rho^2$ vanishes only at the umbilical points of $M$.
We call $M$ is a $k$-extremal submanifold if $M$ is a critical point of $F_k$.

When $k=n/2$, $F_k$ is just the Willmore functional, which is invariant under the conformal transformations of $\mathbb{S}^n$, and the critical points are called Willmore submanifolds.
This case has been intensively studied, see \cites{Li01,Li02,Li02a, GLW01} and the references therein.
Li proved an analogue of Theorem \ref{thm-Sim} for Willmore submanifolds.
\begin{thm}[\cites{Li01,Li02a}]\label{thm-Li}
	Let $M$ be an $n$-dimensional closed Willmore submanifold in $\mathbb {S}^{n+p}$.
	If
	\begin{equation}\label{Li-pinch-cond}
		\rho^2\le \frac{n}{2-1/p},
	\end{equation}
	then either $\rho^2\equiv 0$ and $M$ is totally umbilical, or $\rho^2\equiv \frac{n}{2-1/p}$. In the latter case, $M$ is either one of the Willmore tori with $p=1$, or the Veronese surface in $\mathbb{S}^4$ with $n=p=2$.
\end{thm}
The Willmore tori are
\begin{equation}
	W_{m,n-m}:=\mathbb {S}^{m}(\sqrt{\frac{n-m}{n}})\times \mathbb {S}^{n-m}(\sqrt{\frac{m}{n}}),  \quad 1\le m \le n-1.
\end{equation}
An interesting fact is that a Willmore torus coincides with a Clifford torus if and only if $n$ is even and $m=n/2$.

When $k=1$, $F_k$ is firstly studied by Guo-Li \cite{GL07}, and the critical points are called extremal submanifolds. When $n=2$, $F_1$ becomes the Willmore functional. In \cite{GL07} the authors proved that, for a closed extremal submanifold $M^n$ in $\mathbb{S}^{n+p}$, if $\rho^2\le \frac{n}{2-1/p}$, then either $M$ is totally umbilical, the Clifford torus $C_{m,m}$ with $n=2m$, or the Veronese surface in $\mathbb{S}^4$ with $n=p=2$. Later, Wu \cite{Wu09} extended this rigidity result to $k$-extremal submanifolds for $1<k<n/2$.

When considering the global condition on $\|\rho^2\|_{n/2}$, Yang obtained rigidity theorems for Willmore submanifolds and extremal submanifolds in his doctoral's thesis \cite{Yan10} (also see \cite{XY11}), which were extended to the case $1<k<n/2$ by Guo-Wu \cite{GW15}.

In terms of sectional and Ricci curvatures, Shu \cite{Shu07} and Yang \cite{Yan10} respectively obtained pinching results for Willmore and extremal submanifolds under pointwise conditions; Han improved Shu's sectional curvature pinching bound in his master's thesis \cite{Han13}. Very recently Yang-Fu-Zhang \cite{YFZ22} obtained pinching theorem for both Willmore and extremal submanifolds under integral Ricci curvature condition.

In this paper, we study the rigidity of $k$-extremal submanifolds for general $k\ge 1$, and prove various pinching theorems under different curvature conditions.
Firstly, in terms of sectional curvatures, we prove the following results.
\begin{thm}\label{thm_sec}
	Let $M$ be an $n$-dimensional closed $k$-extremal submanifold in $\mathbb {S}^{n+p}$, $n\ge 2, k\ge 1$.
	If
	\begin{equation}\label{cond-sec}
		\sec^{M}\ge C_1(n,p,H,\rho,k),
	\end{equation}
	where
	\begin{equation}\label{bound-sec}
		C_1(n,p,H,\rho,k)=\frac{p\sgn(p-1)}{2(p+1)}+\frac{n-2}{\sqrt{n(n-1)}}H\rho + \Big(1-\big(\frac{n}{2k}-1\big)\big(1-\frac{p\sgn(p-1)}{2(p+1)}\big)\Big)H^2,
	\end{equation}
	 then  $M$ is either the totally umbilical sphere, or the Clifford torus $C_{m,m}$ in $\mathbb{S}^{n+1}$ with $n=2m$, or the Veronese surface in $\mathbb{S}^4$ with $n=p=2$.
\end{thm}

\begin{thm}\label{thm_sec-n}
	Let $M$ be an $n$-dimensional closed $k$-extremal submanifold in $\mathbb {S}^{n+p}$, $n\ge 2, k\ge 1$.
	If
	\begin{equation}\label{cond-sec-n}
		\sec^{M}\ge C_1'(n,H,\rho,k),
	\end{equation}
	where
	\begin{equation}\label{bound-sec-n}
		C_1'(n,H,\rho,k)=\frac{n}{2(n+1)}+\frac{n-2}{\sqrt{n(n-1)}}H\rho + \Big(1-\big(\frac{n}{2k}-1\big)\frac{n+2}{2(n+1)}\Big)H^2,
	\end{equation}
	then  $M$ is either the totally umbilical sphere, or the Veronese submanifold.
\end{thm}

\begin{cor}\label{cor_sec}
	Let $M$ be an $n$-dimensional closed $k$-extremal submanifold in $\mathbb {S}^{n+p}$, $n\ge 2, k\ge 1$.
	If
	\begin{equation*}
		\sec^{M}>C_1(n,p,H,\rho,k) \mbox{ or } \sec^{M}>C_1'(n,H,\rho,k),
	\end{equation*}
	then $M$ is totally umbilical.
\end{cor}

For integral sectional curvatures, we have

\begin{thm}\label{thm_sec_int}
	Let $M$ be an $n$-dimensional closed $k$-extremal submanifold in $\mathbb {S}^{n+p}$, $n\ge 3, k\ge 1$.
	Given a function $\Phi$ on $M$ such that $\Phi>C_1(n,p,H,\rho,k)$. Then there exists a constant $\epsilon_1(n,p,H,k,\Phi)$ such that if
	\begin{equation*}
		\|\besec\|_{n/2}< \epsilon_1(n,p,H,k,\rho,\Phi),
	\end{equation*}
	then $M$ is totally umbilical.
\end{thm}

\begin{thm}\label{thm_sec_int-n}
	Let $M$ be an $n$-dimensional closed $k$-extremal submanifold in $\mathbb {S}^{n+p}$, $n\ge 3, k\ge 1$.
	Given a function $\Phi$ on $M$ such that $\Phi>C_1'(n,H,\rho,k)$. Then there exists a constant $\epsilon_1'(n,H,k,\Phi)$ such that if
	\begin{equation*}
		\|\besec\|_{n/2}< \epsilon_1'(n,H,k,\rho,\Phi),
	\end{equation*}
	then $M$ is totally umbilical.
\end{thm}

Next, in terms of Ricci curvature, we obtain the following results.
\begin{thm}\label{thm_ricci}
	Let $M$ be an $n$-dimensional closed $k$-extremal submanifold in $\mathbb {S}^{n+p}$, $n\ge 4, k\ge 1$.
	If
	\begin{equation}\label{cond-ric}
	\Ric^{M}\ge C_2(n, H, \rho, k),
	\end{equation}
	where
	\begin{equation}\label{bound-ricci}
		C_2(n, H, \rho, k) = (n-2)+\frac{(n-2)^2}{\sqrt{n(n-1)} } H\rho +n(1-\frac{1}{2k})H^2,
	\end{equation}
	then $M$ is either the totally umbilical sphere, or the Clifford torus $C_{m,m}$ in $\mathbb{S}^{n+1}$ with $n=2m$, or the complex projective space $\mathbb{C}P^2_{4/3}$ in $\mathbb{S}^{7}$.
\end{thm}

\begin{cor}\label{cor_ric}
	Let $M$ be an $n$-dimensional closed $k$-extremal submanifold in $\mathbb {S}^{n+p}$, $n\ge 4, k\ge 1$.
	If
	\begin{equation*}
		\Ric^{M}>C_2(n,H,\rho,k),
	\end{equation*}
	then $M$ is totally umbilical.
\end{cor}

\begin{thm}\label{thm_ricci_int}
	Let $M$ be an $n$-dimensional closed $k$-extremal submanifold in $\mathbb {S}^{n+p}$, $n\ge 4, k\ge 1$.
	Given a function $\Phi$ on $M$ such that $(n-1)\Phi>C_2(n,H,\rho,k)$. Then there exists a constant $\epsilon_2(n,H,k,\Phi)$ such that if
	\begin{equation*}
		\|\beric\|_{n/2}< \epsilon_2(n,H,k,\rho,\Phi),
	\end{equation*}
	then $M$ is totally umbilical.
\end{thm}

At last, we show rigidity theorems in terms of the scalar curvature.

\begin{thm}\label{thm_scal}
	Let $M$ be an $n$-dimensional closed $k$-extremal submanifold in $\mathbb {S}^{n+p}$, $n\ge 2, k\ge 1$.
	If the scalar curvature of $M$ satisfies
	\begin{equation}\label{cond-scal}
	\rho^2\le C_3(n,p,H,k),
	\end{equation}
	where
	\begin{equation}\label{bound-scal}
		C_3(n,p, H,k) = \Big(n+n(\frac{n}{2k}-1)H^2\Big)(1-\frac{1}{3}\sgn(p-1)),
	\end{equation}
	then either $\rho^2\equiv 0$ and $M$ is totally umbilical, or $\rho^2\equiv C_3(n,p, H,k)$.
	In the latter case, $M$ is either one of the $k$-extremal tori $T_{m,k}$ with $p=1$ (see Definition \ref{def-k-tori}), or the Veronese surface in $\mathbb{S}^4$ with $n=p=2$.
\end{thm}

The following is a version of integral curvatures.
\begin{thm}\label{thm_scal_int}
	Let $M$ be an $n$-dimensional closed $k$-extremal submanifold in $\mathbb {S}^{n+p}$, $n\ge 2, k\ge 1$.
	Given a function $\Phi$ on $M$ such that $\Phi<C_3(n,p, H,k)$. Then there exists a constant $\epsilon_3(n,p, H,k,\Phi)$ such that if
	\begin{equation*}
		\|(\rho^2-\Phi)_{+}\|_{n/2}< \epsilon_3(n,p, H,k,\Phi),
	\end{equation*}
	then $M$ is totally umbilical.
\end{thm}
When $k<n/2$, we can choose an upper bound independent of $H$.
\begin{thm}\label{thm_scal_int'}
	Let $M$ be an $n$-dimensional closed $k$-extremal submanifold in $\mathbb {S}^{n+p}$, $n\ge 2,  1\le k<n/2$.
	Given a function $\Phi$ on $M$ such that $\Phi<(1-\frac{1}{3}\sgn(p-1))n$. Then there exists a constant $\epsilon_3'(n,p,k,\Phi)$ such that if
	\begin{equation*}
		\|(\rho^2-\Phi)_{+}\|_{n/2}< \epsilon_3'(n,p,k,\Phi),
	\end{equation*}
	then $M$ is totally umbilical.
\end{thm}

We give some comments on our results.
\begin{enumerate}
	\item Corollary \ref{cor_ric} can be derived from Theorem \ref{thm_ricci_int}; Corollary \ref{cor_sec} can be derived from Theorems \ref{thm_sec_int} and \ref{thm_sec_int-n} when $n\ge 3$.
	\item For the totally umbilical sphere $\mathbb{S}^n(r)$, we have $1/r^2=1+H^2$, and the sectional curvature is $1+H^2$ and the Ricci curvature is $(n-1)(1+H^2)$. If $1\le k\le n/2$, then \eqref{cond-sec} and \eqref{cond-ric} are always satisfied. But if $k>n/2$, then the radius $r$ of the sphere is restricted.
	\item Shu's Ricci pinching theorem in \cite{Shu07} for Willmore submanifolds requires $n\ge 5$,
	while ours requires $n\ge 4$.
	Even for $k=n/2$ and $k=1$, our Ricci curvature pinching bounds $C_2$ is better than the ones in \cites{Shu07,Yan10,Han13}.
	\item Our sectional curvature pinching bounds $C_1$ is better than the one in \cite{Yan10} for $k=1$,
	and coincides with the one in \cite{Han13} for $k=n/2$.
	However, the conclusions of the sectional curvature pinching theorems in \cites{Shu07, Han13} for $k=n/2$ say that one of the candidates of $M$ is the Willmore torus $W_{1,n-1}$;
	this is impossible when $n\ge 3$ since $K_{\min}=0$ but $C_1>0$.
	The correct candidate is $C_{m,m}$.
	\item Theorem \ref{thm_scal} recovers Theorem \ref{thm-Li} when $k=n/2$; it also improves Theorem 1.3 in \cite[]{GL07} when $k=1$.
	Compared with Theorem 5.1 in  \cite{Wu09}, our bound $C_3$ is a bit better, and we determine the tori when $p=1$.
	\item If we take $\Phi\equiv 0$, then Theorem \ref{thm_scal_int} recovers Theorem 2.2 in \cite{Yan10} for $k=n/2$, while Theorem \ref{thm_scal_int'} recovers Theorem A in \cite{XY11} for $k=1$ and Theorem 4 in \cite{GW15} for $1<k<n/2$.
\end{enumerate}

This paper is organized as follows.
In Section 2, we introduce notations, recall the basic formulas in submanifolds of geometry, and list some lemmas.
In Section 3, we show some examples of $k$-extremal submanifolds in a sphere, including minimal Einstein submanifolds and isoparametric hypersurfaces.
In the last three sections, we prove our main results under sectional, Ricci and scalar curvatures conditions respectively.

\section{Preliminaries}
\subsection{Notations and fundamental formulas for submanifolds}
Firstly, we introduce some notations and recall the fundamental formulas for submanifolds.
Let $M$ be an $n$-dimensional submanifold in $\mathbb{S}^{n+p}$. 
We choose a local orthonormal frame $\{e_1,\cdots,e_{n+p}\}$ such that $\{e_1,\cdots,e_{n}\}$ are tangent to $M$ and $\{e_{n+1},\cdots,e_{n+p}\}$ are normal to $M$ when restricted to $M$.
Let $\{\omega_A\}$ be the dual coframe.
We make the following convention on the range of indices except special declaration:
\begin{equation*}
	1\leq A, B, C,\cdots\leq n+p; \quad 1\leq i, j, k,\cdots\leq n;\quad  n+1\leq\alpha, \beta, \gamma,\cdots\leq n+p.
\end{equation*}
Denote
\begin{equation*}
	h=\sum_{i,j,\alpha}h_{ij}^{\alpha}\omega_i\otimes\omega_j\otimes
	e_{\alpha}
\end{equation*}
the second fundamental form of $M$ in $\mathbb{S}^{n+p}$, and set
\begin{align*}
	A_{\alpha}=&(h_{ij}^{\alpha}), \quad H^{\alpha}=\frac{1}{n}\sum_{k}h_{kk}^\alpha,\quad \mathbf{H}=\sum_{\alpha}H^\alpha e_\alpha,\quad H=|\mathbf{H}|=\sqrt{\sum_{\alpha}(H^{\alpha})^2},\quad S=\sum_{i,j,\alpha}(h_{ij}^{\alpha})^2.
\end{align*}
Now Gauss, Codazzi and Ricci equations are respectively given by
\begin{align}
&R_{ijkl}=(\delta_{ik}\delta_{jl}
-\delta_{il}\delta_{jk})+\sum_{\alpha}(h_{ik}^{\alpha}h_{jl}^{\alpha}-h_{il}^{\alpha}h_{jk}^{\alpha}),\label{eqG}\\
&h_{ijk}^{\alpha}=h_{ikj}^{\alpha},\label{eqC}\\
&R_{\alpha\beta
	ij}^{\bot}=\sum_{k}h_{ik}^{\alpha}h_{kj}^{\beta}-\sum_{k}h_{ik}^{\beta}h_{kj}^{\alpha},\label{eqR}
\end{align}
where $R_{ijkl}$ and $h_{ijk}^{\alpha}$ are the components of  Riemannian curvature of $M$ and covariant derivative of $ h_{ij}^{\alpha}$ under the orthonormal frame respectively. The Ricci identity shows that
\begin{equation}\label{eq-Ric-id}
h_{ijkl}^{\alpha}-h_{ijlk}^{\alpha}=\sum_{m}R_{mikl}h_{mj}^{\alpha}+\sum_{m}R_{mjkl}h_{im}^{\alpha}+\sum_{\beta}R_{\beta\alpha kl}^{\bot}h_{ij}^{\beta}.
\end{equation}
From (\ref{eqG}), we can get the Ricci curvature and the scalar curvature respectively as follows:
\begin{align}
R_{ij}&=(n-1)\delta_{ij}+n\sum_{\alpha}H^{\alpha}h^{\alpha}_{ij}-\sum_{\alpha, k}h^{\alpha}_{ik}h^{\alpha}_{kj},\label{Ricci}\\
R&=n(n-1)+n^2H^2-S.\label{scalar}
\end{align}
Since $S \ge nH^2$, we have $R \le n(n-1)(1+H^2)$. When $H =0,\ \Ric \le n-1$.

When studying the rigidity of a non-minimal submanifolds, we usually use the following traceless second fundamental form
\begin{align}
	\tilde h=\sum_{i,j}\tilde h_{ij}^\alpha\omega_{i}\otimes \omega_{j} \otimes e_{\alpha },
\end{align}
where $\tilde h_{ij}^\alpha = h_{ij}^\alpha-H^\alpha \delta_{ij}$.
We set
\begin{align}\label{twosigma}
	\tilde{A}_{\alpha}=(\tilde h_{ij}^\alpha),\quad \sigma_{\alpha \beta }=\sum_{i,j}h_{ij}^{\alpha }h_{ij}^{\beta}=\tr(A_{\alpha}A_{\beta}),\quad
	\tilde \sigma_{\alpha \beta }=\sum_{i,j}\tilde h_{ij}^{\alpha }\tilde h_{ij}^{\beta}=\tr(\tilde{A}_{\alpha}\tilde{A}_{\beta}).
\end{align}
Then it can be easily checked that
\begin{gather}
	\rho^2=\sum_{i,j,\alpha}(\tilde{h}_{ij}^{\alpha})^{2}=S-nH^2,\\
	\sigma_{\alpha \beta }=
	\tilde \sigma_{\alpha \beta }+nH^\alpha H^\beta,\label{sigma-1}\\
	\sum_{\alpha ,\beta }H^\alpha H^\beta \sigma_{\alpha \beta }
	=\sum_{\alpha ,\beta }H^\alpha H^\beta \tilde \sigma_{\alpha \beta }+nH^4,\label{sigma}\\
	\sum_{\alpha ,\beta}H^\alpha \tr(A_{\alpha}A_{\beta}^2) =
	\sum_{\alpha ,\beta}H^\alpha \tr(\tilde{A}_{\alpha}\tilde{A}_{\beta}^2)
	+2\sum_{\alpha ,\beta }H^\alpha H^\beta \tilde \sigma_{\alpha \beta}+H^2\rho^2+nH^4.\label{threeitems}
\end{gather}

\subsection{Some Lemmas}
We recall some results which will be used to prove the main theorems.
The following three algebraic lemmas can be proven by using the method of Lagrange multipliers.

\begin{lem}[{cf. \cite{Oku74, Li96}}]\label{ineqs-hypersurface}
	Let $a_{i}$ for $i=1, \cdots, n$ be real numbers satisfying $\sum_{i=1}^{n}a_{i}=0$ and $\sum_{i=1}^{n}a_{i}^{2}=a$. Then
	\begin{align}
		\Big|\sum_{i=1}^{n}a_{i}^3\Big|\le
		\frac{n-2}{\sqrt{n(n-1)}}a^{3/2}.
	\end{align}
	Moreover, equality holds if and only if at least $(n-1)$ of the $a_i$'s are equal.
\end{lem}

\begin{lem}[{cf. \cite[Lemma 3.4]{Che02}}]\label{ineqs}
	Let $a_{i}$ and $b_{i}$ for $i=1, \cdots, n$ be real numbers satisfying $\sum_{i=1}^{n}a_{i}=0$ and $\sum_{i=1}^{n}a_{i}^{2}=a.$ Then
	\begin{align}
		\Big|\sum_{i=1}^{n}a_{i}{b_{i}}^2\Big|\le \sqrt{\sum_{i=1}^{n}{b_{i}}^4-\frac{(\sum_{i=1}^{n} b_{i}^{2})^{2}}{n}}\sqrt{a}.
	\end{align}
\end{lem}
\begin{lem}[{cf. \cite[Lemma 3.3]{Che02}}]\label{ineq}
	Let $b_{i}$ for $i=1, \cdots, n$ be real numbers satisfying $\sum_{i=1}^{n}b_{i}=0$ and $\sum_{i=1}^{n}b_{i}^{2}=B.$ Then
	\begin{align}
		\sum_{i=1}^{n}{b_{i}}^4-\frac{B^2}{n}\le \frac{(n-2)^2}{n(n-1)}B^2.
	\end{align}
\end{lem}

Ge-Tang \cite{GT08} and Lu \cite{Lu11} proved the following DDVV inequality independently and differently.
\begin{lem}[\cite{GT08, Lu11}]\label{lem-DDVV}
	Let $B_{1},\cdots,B_{p}$ be symmetric $(n\times n)-$matrices ($p\ge 2$). Then
	\begin{align}\label{DDVVineq}
		\sum_{r,s=1}^{p} N([B_{r}, B_{s}])\le \big(\sum_{r=1}^{p} N(B_{r})\big)^{2},
	\end{align}
	where the equality holds if and only if under some rotation all $B_{r}^{'}s$ are zero except two matrices, which can be written as
	\begin{align*}
		\tilde {B}_{1}=P\begin{pmatrix}
			0&  \mu &  0&  \cdots &0 \\
			\mu & 0 &  0&  \cdots &0 \\
			0& 0 & 0 & 0 & 0\\
			\vdots & \vdots  &\vdots   &\ddots   &\vdots  \\
			0& 0 &  0&  \cdots &0
		\end{pmatrix}
		P^t,\quad
		\tilde {B}_{2}=P\begin{pmatrix}
			\mu &  0 &  0&  \cdots &0 \\
			0 & -\mu  &  0&  \cdots &0 \\
			0& 0 & 0 & 0 & 0\\
			\vdots & \vdots  &\vdots   &\ddots   &\vdots  \\
			0& 0 &  0&  \cdots &0
		\end{pmatrix}
		P^t,
	\end{align*}
	for some  $P\in O(n)$.
\end{lem}

There are other estimates for $\sum_{r,s}N([B_r, B_s])$.
\begin{lem}[cf. \cite{Ito75}*{Proposition 1}]\label{lem-Itoh}
	Let $B_{1},\cdots,B_{p}$ be symmetric $(n\times n)-$matrices. Then
	\begin{align}\label{Itoh-ineq}
		\sum_{r,s=1}^{p} N([B_{r}, B_{s}])\le n\sum_{r,s=1}^{p} \big(\tr(B_rB_s)\big)^2.
	\end{align}
\end{lem}

\begin{lem}[cf. \cite{Sim68}*{Theorem 1} and \cite{LL92}*{Proposition 1}]\label{lem-LL}
	Let $B_{1},\cdots,B_{p}$ be symmetric $(n\times n)-$matrices. Then
	\begin{align}\label{LL-ineq}
		\sum_{r,s=1}^{p} N([B_{r}, B_{s}])+\sum_{r,s=1}^{p} \big(\tr(B_rB_s)\big)^2 \le \big(1+\frac{1}{2}\sgn(p-1)\big) \big(\sum_{r=1}^{p} N(B_{r})\big)^{2},
	\end{align}
	where the equality holds if and only if either $p=1$, or $p\ge 2$ and under some rotation all $B_{r}^{'}s$ are zero except two matrices, which can be written as
	\begin{align*}
		\tilde {B}_{1}=P\begin{pmatrix}
			0&  \mu &  0&  \cdots &0 \\
			\mu & 0 &  0&  \cdots &0 \\
			0& 0 & 0 & 0 & 0\\
			\vdots & \vdots  &\vdots   &\ddots   &\vdots  \\
			0& 0 &  0&  \cdots &0
		\end{pmatrix}
		P^t,\quad
		\tilde {B}_{2}=P\begin{pmatrix}
			\mu &  0 &  0&  \cdots &0 \\
			0 & -\mu  &  0&  \cdots &0 \\
			0& 0 & 0 & 0 & 0\\
			\vdots & \vdots  &\vdots   &\ddots   &\vdots  \\
			0& 0 &  0&  \cdots &0
		\end{pmatrix}
		P^t,
	\end{align*}
	for some  $P\in O(n)$.
\end{lem}

By using a Sobolev inequality in \cite{HS74}, we have
\begin{prop}[cf. \cites{Xu94,CW21}]\label{prop-int}
	Let $M$ be an $n(\ge 3)$-dimensional closed submanifold in $\mathbb{S}^{n+p}$. Then for all $t>0$ and $f\in C^1(M),$  $f\geq 0$, we have
	\begin{equation}\label{Sineq}
		\int_M|\nabla f|^2\geq \mathfrak{A}(n,t)\|f^2\|_{{n}/{n-2}}-\mathfrak{B}(n,t)\int_M (1+H^2)f^2,
	\end{equation}
	where
	\begin{equation}\label{const-def-ABC}
		\mathfrak{A}(n,t)=\frac{(n-2)^2}{4(n-1)^2(1+t)}\frac{1}{C^2(n)}, 
		\quad \mathfrak{B}(n,t)=\frac{(n-2)^2}{4(n-1)^2t} , \quad C(n)=2^{n}\frac{(n+1)^{1+1/n}}{(n-1)\omega_n^{1/n}},
	\end{equation}
	and $\omega_n$ is the volume of the unit ball in $\mathbb{R}^n$.
\end{prop}

The following is a Kato-type inequality.
\begin{lem}[cf. \cites{Xu94,XY11}]\label{gradhrho}
	Let $M$ be an $n$-dimensional submanifold in $\mathbb{S}^{n+p}.$ Then for any $\epsilon\neq 0$, we have
	\begin{align}\label{rho}
 |\nabla \tilde h |^2 \ge
\frac{n+2}{n}\left | \nabla \rho _{\epsilon } \right |^2,
\end{align}
	where $ |\nabla \tilde{h} |^{2}=\sum_{\alpha ,i,j,k}(\tilde h_{ijk}^{\alpha })^{2},  \rho_{\epsilon }=(\rho^2+np\epsilon^2)^{1/2} >0$.
\end{lem}
Since $|\nabla \rho_\epsilon^{k}|^{2} =k^2 \rho_\epsilon^{2k-2}|\nabla \rho_\epsilon|^2$, it follows from  \eqref{Sineq} and \eqref{rho} that
\begin{equation}\label{ineqrho}
	k^2\int_{M} \rho_\epsilon^{2k-2}|\nabla \tilde h|^2
	\ge \frac{n+2}{n}\int_M |\nabla \rho_\epsilon^{k}|^{2}\ge  \frac{n+2}{n}\left \{ \mathfrak{A}(n,t)\|\rho_\epsilon^{2k}\|_{{n}/{n-2} }
	-\mathfrak{B}(n,t)\int_{M}(1+H^2) \rho_\epsilon^{2k}\right \}.
\end{equation}

\section{Examples of $k$-extremal submanifolds}
Wu (cf. \cite[Theorem 1.1]{Wu09}) calculated the Euler-Lagrangian equation of $F_k$ and showed that
$M^n$ is an $k$-extremal submanifold of $\mathbb{S}^{n+p}$ if and only if 
	\begin{equation}\label{kextremaleq}
		\begin{aligned}
			&\phantom{+}\rho^{2k-2}\big[\sum_{\beta}\tr(A_\alpha A_\beta^2)-
			\sum_{\beta}H^{\beta }\sigma_{\alpha\beta }-\frac{n}{2k}\rho^2H^{\alpha } \big]\\
			&+(n-1)H^{\alpha }\Delta (\rho^{2k-2})+2(n-1)\sum_{i}(\rho^{2k-2})_{i}H_{i}^{\alpha }\\
			&+(n-1)\rho^{2k-2}\Delta^{\perp }H^{\alpha }-\square^\alpha(\rho^{2k-2})=0
		\end{aligned}
	\end{equation}
	holds for $n+1\le \alpha \le n+p$.
	Here $\square^\alpha: C^{\infty} \to \mathbb{R}$ defined by
	\begin{equation}
		\square^\alpha f=(nH^\alpha \delta_{ij}-h_{ij}^\alpha)f_{ij}.
	\end{equation}
	When $M$ is closed, $\square^\alpha$ is a self-adjoint operator (cf. \cite{CY77}), i.e., \begin{align}\label{selfadjoint}
		\int_{M}g \square^\alpha f=\int_{M}f\square^\alpha g.
	\end{align}

Special cases of \eqref{kextremaleq} were obtained by Li \cite{Li01} for $k=n/2$ and by Guo-Li \cite{GL07} for $k=1$, respectively.

\subsection{Einstein submanifolds}
Guo-Li-Wang (\cite[Theorem 4.2]{GLW01}) showed that any minimal Einstein submanifold in the unit sphere must be a Willmore submanifold. Similarly, we can prove
\begin{prop}\label{minieinstein}
	Let $M$ be an $n$-dimensional minimal submanifold in $\mathbb{S}^{n+p}$. If either $n\ge 3$ and $M$ is Einstein, or $n=2$ and $M$ has constant Gauss curvature, then $M$ is $k$-extremal for $k\ge 1$.
\end{prop}
\begin{proof}
	From \eqref{Ricci} one can check that
	\begin{equation}\label{unfoldrho}
		\sum_{\beta}\tr(A_\alpha A_\beta^2)-\sum_{\beta}H^\beta \sigma_{\alpha\beta}=\sum_{i,j}h_{ij}^\alpha R_{ij}-n(n-1)H^\alpha-(n-1)\sum_{\beta}H^\beta \sigma_{\alpha\beta}.
	\end{equation}
	Either $n\ge 3$ and $M$ is Einstein, or $n=2$ and $M$ has constant Gauss curvature, we always have $R_{ij}=\frac{R}{n}\delta_{ij}$ with constant $R$.
	Since $M$ is minimal, $\rho^2=S-nH^2=S=n(n-1)-R$ is also constant from \eqref{scalar}.
	By using \eqref{unfoldrho}, one can easily verify that \eqref{kextremaleq} holds.
\end{proof}
\begin{rem}
	When considering Willmore surface ($n=2, k=1$), we can remove the assumption of ``constant Gauss curvature'', that is, any minimal surface in $\mathbb{S}^{2+p}$ is a Willmore surface (cf. \cite[Remark 4.2]{GLW01}).
\end{rem}

\begin{ex}[Projective spaces]
Let $\mathbb{F}P^m$ be the projective space $\mathbb{F}P^m$ of real dimension $n=m\cdot d_{\mathbb{F}}$.
Here $\mathbb{F}$ denotes the field $\mathbb{R}$ of real numbers, the field $\mathbb{C}$ of complex numbers or the field $\mathbb{Q}$ of quaternions,
and $d_{\mathbb{R}}=1, d_{\mathbb{C}}=2$ and $d_{\mathbb{Q}}=4$.

It is well known that there exists a minimal embedding $\phi_1$, called the first standard embedding, of $\mathbb{F}P^m$ into $\mathbb{S}^{N}$ with $N=\frac{m(m+1)}{2}d+m-1$ (cf. \cite{Che73}*{Chapter 4.6}).

In these cases, $\mathbb{R}P^n$ is just the Veronese submanifold with constant sectional curvature $\frac{n}{2(n+1)}$;
$\mathbb{C}P^m$ are Einstein with constant holomorphic sectional curvature $\frac{2m}{m+1}$ and with constant Ricci curvatures $m$;
$\mathbb{Q}P^m$ is also Einstein with constant Ricci curvaturs $\frac{2m(m+2)}{m+1}$.

By Proposition \ref{minieinstein}, we know that $
\phi_1:\mathbb{F}P^m\to \mathbb{S}^N$ is $k$-extremal.
\end{ex}

\subsection{Isoparametric hypersurfaces}
Next we consider the isoparametric hypersurfaces in a sphere. In this case, all the principal curvatures $\lambda_i$ are constant, so are $\rho^2$ and $H$. We immediately derive
\begin{lem}
Let $M$ be an isoparametric hypersurface in $\mathbb{S}^{n+1}$. If $M$ is $k$-extremal ($k\ge 1$), then
\begin{equation}\label{eq-isoparametric}
	\sum_{i}\lambda_i^3+\frac{n^2}{2k}H^3-\Big(\frac{n}{2k}+1\Big)HS=0,
\end{equation}
where $S=\sum_{i}\lambda_i^2, H=\frac{1}{n}\sum_{i}\lambda_i$.
\end{lem}

The classification of isoparametric hypersurfaces in a sphere is an important problem in geometry. It has a long history, and recently was solved completely after a series of work, see a survey article \cite{Chi18} and the references therein.
We briefly summarize some key properties of the isoparametric theory in a sphere.
\begin{lem}[\cites{Car38, Abr83, Muen80, Muen81,Sto99}]\label{lem-iso-class}
	Let $M$ be an $n$-dimension closed isoparametric hypersurface in $\mathbb{S}^{n+1}$.
	Let $\lambda_1>\dots>\lambda_g$ be the distinct principal curvatures with multiplicities $m_1,\dots,m_g$. Then we have
	\begin{enumerate}
		\item $g=1,2,3,4$ or $6$.
		\item If $g=1$, then $M$ is totally umbilical.
		\item If $g=2$, then $M=\mathbb{S}^m(a)\times \mathbb{S}^{n-m}(\sqrt{1-a^2})$.
		\item If $g=3$, then $m_1=m_2=m_3=2^l (l=0,1,2,3)$.
		\item If $g=4$, then $m_1=m_3, m_2=m_4$.
				Moreover, $(m_1,m_2)=(2,2)$ or $(4,5)$, or $m_1+m_2+1$ is a multiple of $2^{\xi(m_1-1)}$, where $\xi(l)$ denotes the numbers of integers $s$ such that $1\le s\le l$ and $s\equiv 0,1,2,4\ (\mathrm{mod\ } 8)$.
		\item If $g=6$, then $m_1=\dots=m_6=1$ or $2$.
		\item There exists an angle $\theta\in (0,\pi/g)$ such that
		\begin{equation}
			\lambda_{\alpha}=\cot(\theta+\frac{\alpha-1}{g}\pi), \quad \alpha=1,\dots,g.
		\end{equation}
	\end{enumerate}
\end{lem}

Clearly, the totally hypersurfaces ($g=1$) are always $k$-extremal. Now we discuss $g=2,3,4$ and $6$, respectively.

\begin{defn}\label{def-k-tori}
	Given $k\ge 1$, we denote an $n$-dimensional torus $T_{m,k}$ as follows.
	\begin{enumerate}
		\item When either $k>n/4$ and there exists $m (1\le m\le n-1)$ such that $n-2k<m<2k$, or $k<n/4$ and there exists $m (1\le m\le n-1)$ such that $2k<m<n-2k$, define
		\begin{equation*}
			T_{m,k}=\mathbb{S}^m(\sqrt{\frac{m-2k}{n-4k}})\times \mathbb{S}^{n-m}(\sqrt{\frac{n-m-2k}{n-4k}});
		\end{equation*}

		\item When $k=n/4$, and $n=2m$ is even, define
		\begin{equation*}
			T_{m,k}=\mathbb{S}^m({\frac{1}{\sqrt{2}}})\times \mathbb{S}^m({\frac{1}{\sqrt{2}}}).
		\end{equation*}
	\end{enumerate}
	We call $T_{m,k}$ the $k$-extremal torus.
\end{defn}

\begin{thm}\label{thm-iso-2}
	Let $M^n$ be a closed isoparametric hypersurface in $\mathbb{S}^{n+1}$ with $g=2$. Given $k\ge 1$.
	Then $M$ is $k$-extremal if and only if $M$ is one of the $k$-extremal tori $T_{m,k}$.
\end{thm}
\begin{proof}
	Since $M=\mathbb{S}^{m}(a)\times \mathbb{S}^{n-m}(\sqrt{1-a^2})$,
	by a choice of unit normal, the principal curvatures are given by $\lambda=-\frac{\sqrt{1-a^2}}{a}$ with multiplicity $m$ and $\mu=\frac{a}{\sqrt{1-a^2}}$ with multiplicity $n-m$. \eqref{eq-isoparametric} becomes
	\begin{equation}
		\Big(m\lambda^3+(n-m)\mu^3\Big)+\frac{1}{2kn}\Big(m\lambda+(n-m)\mu\Big)^3-\frac{n+2k}{2kn}\Big(m\lambda+(n-m)\mu\Big)\Big(m\lambda^2+(n-m)\mu^2\Big)=0.
	\end{equation}
	By using $\lambda\mu=-1$, we obtain
	\begin{equation}
		(m-2k)\lambda^6+(3m-n-2k)\lambda^4+(3m+2k-2n)\lambda^2+(m+2k-n)=0,
	\end{equation}
	equivalently,
	\begin{equation}
		\Big((m-2k)\lambda^2-(n-m-2k)\Big)(\lambda^2+1)^2=0.
	\end{equation}
	If $k=n/4$, then $(m-n/2)(\lambda^2+1)=0$, and we must have $m=n/2=2k$. So $M=C_{m,m}=\mathbb{S}^m({\frac{1}{\sqrt{2}}})\times \mathbb{S}^m({\frac{1}{\sqrt{2}}})$ with $k=m/2=n/4$.

	If $k=m/2$, then we must have $n=m+2k=2m$, i.e., $k=n/4$.
	Therefore, if $k\neq n/4$, then $k\neq m/2$.
	We solve out $\lambda^2=\frac{1-a^2}{a^2}=\frac{n-m-2k}{m-2k}$ and $a^2=\frac{m-2k}{n-4k}$.
	Since $0<a^2<1$, if $k<n/4$, then $m>2k$ and $n-m>2k$, i.e., $2k<m<n-2k$; if $k>n/4$, then $m<2k$ and $n-m<2k$, i.e., $n-2k<m<2k$.
\end{proof}

\begin{rem}\label{rem-k-tori}
	For the existence of $T_{m,k}$, we have the following observation.
	\begin{enumerate}
		\item When $k>(n-1)/2$, $T_{m,k}$ is defined for each $m=1,\dots,n-1$.
		In particular, $T_{m,n/2}=W_{m,n-m}$.

		\item When $k\le (n-1)/2$, $m$ cannot run over $\{1,...,n-1\}$, even doesn't exist for some pair $(n,k)$.

		\item When $n=2m$ is even, $T_{m,k}=C_{m,m}$ is always defined for any $k\ge 1$, and it is the only minimal $k$-extremal torus.
	\end{enumerate}

	A direct computation shows that for each existing $T_{m,k}$, we have
	\begin{equation}
		\rho^2=n+n(\frac{n}{2k}-1)H^2.
	\end{equation}
\end{rem}

\begin{thm}\label{thm-iso-3}
	Let $M^n$ be an isoparametric hypersurface in $\mathbb{S}^{n+1}$ with $g=3$. Given $k\ge 1$.
	\begin{enumerate}
		\item If $n=3$ or $k\neq n/6 (n=6,12,24)$, then $M$ is $k$-extremal if and only if
		\begin{equation*}
			\lambda_1=\sqrt{3},\quad \lambda_2=0, \quad \lambda_3=-\sqrt{3}.
		\end{equation*}
		\item If $k=n/6 (n=6,12,24)$, then $M$ is always $k$-extremal.
	\end{enumerate}
\end{thm}
\begin{proof}
	Let $m:=m_1=m_2=m_3$,
	\begin{equation}\label{eq-principal-curvatures-3}
		\lambda_1=\cot\theta, \quad\lambda_2=\cot(\theta+\frac{\pi}{3})=\frac{\lambda_1-\sqrt{3}}{1+\sqrt{3} \lambda_1}, \quad\lambda_3=\cot(\theta+\frac{2\pi}{3})=\frac{\lambda_1+\sqrt{3}}{1-\sqrt{3} \lambda_1}.
	\end{equation}
	Putting \eqref{eq-principal-curvatures-3} into \eqref{eq-isoparametric} and noticing $n=3m,\lambda_1>1/\sqrt{3}$, we obtain
	\begin{equation}
		(m-2k)(\lambda_1^2-3)(\lambda_1^2+1)^3=0.
	\end{equation}
	When $k\neq m/2$, we have $\lambda_1=\sqrt{3}, \lambda_2=0, \lambda_3=-\sqrt{3}$. This case is minimal.

	When $k=m/2$, there are no extra restrictions on $\lambda_1$. Since $k\ge 1$, this case occurs only when $n=6k=6,n=6k=12$ or $n=6k=24$.
\end{proof}

\begin{thm}\label{thm-iso-6}
	Let $M^n$ be an isoparametric hypersurface in $\mathbb{S}^{n+1}$ with $g=6$. Given $k\ge 1$.
	\begin{enumerate}
		\item If $n=6$, or $n=12$ and $k>1$, then $M$ is $k$-extremal  if and only if
		\begin{equation*}
			\lambda_1=2+\sqrt{3}, \quad \lambda_2=1, \quad
			\lambda_3=2-\sqrt{3}, \quad \lambda_4=-(2-\sqrt{3}),\quad
			\lambda_5=-1, \quad \lambda_6=-(2+\sqrt{3}).
		\end{equation*}
		\item If $n=12$ and $k=1$, then $M$ is always  $k$-extremal.
	\end{enumerate}
\end{thm}
\begin{proof}
	Let $m:=m_1=m_2=m_3=m_4=m_5=m_6\in {1,2}$,
	\begin{equation}\label{eq-principal-curvatures-6}
		\begin{gathered}
			\lambda_1=\cot\theta,\quad
			\lambda_2=\frac{\sqrt{3}\lambda_1-1}{\sqrt{3}+\lambda_1},\quad
			\lambda_3=\frac{\lambda_1-\sqrt{3}}{1+\sqrt{3} \lambda_1}, \\
			\lambda_4=-\frac{1}{\lambda_1},\quad
			\lambda_5=\frac{\lambda_1+\sqrt{3}}{1-\sqrt{3} \lambda_1},\quad
			\lambda_6=\frac{\sqrt{3} \lambda_1+1}{\sqrt{3}-\lambda_1}.
		\end{gathered}
	\end{equation}
	Putting \eqref{eq-principal-curvatures-6} into \eqref{eq-isoparametric} and noticing $n=6m,\lambda_1>\sqrt{3}$, we obtain
	\begin{equation}
		(m-2k)(\lambda_1^2-1)(1-4\lambda_1+\lambda_1^2)(1+4\lambda_1+\lambda_1^2)(1+\lambda_1^2)^6=0.
	\end{equation}

	When $k\neq m/2$, we have
	\begin{equation*}
		\lambda_1=2+\sqrt{3}, \quad \lambda_2=1, \quad \lambda_3=2-\sqrt{3}, \quad \lambda_4=-(2-\sqrt{3}),\quad
		\lambda_5=-1, \quad \lambda_6=-(2+\sqrt{3}).
	\end{equation*}
	This case is minimal.

	When $k=m/2$, there are no extra restrictions on $\lambda_1$. Since $k\ge 1$, this case occurs only when $n=12k=12$.
\end{proof}

\begin{thm}\label{thm-iso-4}
	Let $M^n$ be an isoparametric hypersurface in $\mathbb{S}^{n+1}$ with $g=4$. Given $k\ge 1$.
	Then
	$M$ is a $k$-extremal hypersurface
	if and only if the equation
	\begin{equation*}
		(2k-m_1)m_1(m_1+2m_2)x^2+4m_1m_2(m_2-m_1)x-16(2k-m_2)m_2(2m_1+m_2)=0
	\end{equation*}
	has a positive solution $x=A^2 (A>0)$, and the principal curvatures are given by
	\begin{equation}
		\lambda_1=\lambda,\quad
		\lambda_2=\frac{\lambda-1}{\lambda+1}, \quad
		\lambda_3=-\frac{1}{\lambda},\quad
		\lambda_4=-\frac{\lambda+1}{\lambda-1},
\end{equation}
where $\lambda>1$ satisfies $\lambda-1/\lambda=A>0$.
Moreover,
	\begin{enumerate}
		\item When $m_1=m_2=n/4$, if $k=n/8$, then $M$ is always $k$-extremal;
		if $k\ne n/8$, then
		\begin{equation}\label{eq-4-equ}
			\lambda_1=\sqrt{2}+1, \quad \lambda_2=\sqrt{2}-1, \quad \lambda_3=1-\sqrt{2}, \quad \lambda_4=-(1+\sqrt{2}).
		\end{equation}
		\item When $m_1<m_2$, if $m_1\le k\le m_2$, then $M$ is not $k$-extremal;
		if $k<m_1$ or $k>m_2$, then \eqref{eq-4-equ} has exactly a positive solution.
	\end{enumerate}
\end{thm}
\begin{proof}
	We have $m_1=m_3, m_2=m_4, n=2(m_1+m_2)$,
	\begin{equation}\label{eq-principal-curvatures-4}
			\lambda_1=\cot\theta,\quad
			\lambda_2=\frac{\lambda_1-1}{\lambda_1+1}, \quad
			\lambda_3=-\frac{1}{\lambda_1},\quad
			\lambda_4=-\frac{\lambda_1+1}{\lambda_1-1}=-\frac{1}{\lambda_2}.
	\end{equation}
	Set $A=\lambda_1-\frac{1}{\lambda_1}, B=\lambda_2-\frac{1}{\lambda_2}$,
	then
	\begin{equation}\label{eq-principal-curvatures-4-relation}
		nH=m_1A+m_2B,\quad S=m_1A^2+m_2B^2+n, \quad \sum_{i}\lambda^3_i=m_1A^3+m_2B^3+3nH.
	\end{equation}

	Putting \eqref{eq-principal-curvatures-4-relation} into \eqref{eq-isoparametric} and noticing $AB=-4$, we obtain
	\begin{equation}\label{eq-principal-curvatures-4'}
		(2k-m_1)m_1(m_1+2m_2)A^4+4m_1m_2(m_2-m_1)A^2-16(2k-m_2)m_2(2m_1+m_2)=0.
	\end{equation}

	\textbf{Case 1: $m_1=m_2$}. In this case,  $m_1=m_2=m=n/4$ and \eqref{eq-principal-curvatures-4'} becomes
	\begin{equation}
		(2k-m)(A^2-4)(A^2+4)=0.
	\end{equation}
	When $k\neq m/2$, we have $\lambda_1-1/\lambda_1=2$, and then
	\begin{equation*}
		\lambda_1=\sqrt{2}+1, \quad \lambda_2=\sqrt{2}-1, \quad \lambda_3=1-\sqrt{2}, \quad \lambda_4=-(1+\sqrt{2}).
	\end{equation*}
	This case is minimal.

	When $k=m/2=n/8$, there are no extra restrictions on $A$ (so on $\lambda_1$).

	\textbf{Case 2: $m_1<m_2$}. In this case, if we denote $b=m_2/m_1$, then \eqref{eq-principal-curvatures-4'} becomes
	\begin{equation}
		\Big(\frac{2k}{m_1}-1\Big)(1+2b)A^4+4b(b-1)A^2-16\Big(\frac{2k}{m_1}-b\Big)b(2+b)=0.
	\end{equation}
	We consider the quadratic polynomial
	\begin{equation}
		P_{b,k}(x)=\Big(\frac{2k}{m_1}-1\Big)(1+2b)x^2+4b(b-1)x-16\Big(\frac{2k}{m_1}-b\Big)b(2+b).
	\end{equation}
	When $1\le 2k/m_1\le b$, $P_{b,k}$ has no positive root since $b>1$.
	When $2k/m_1<1$ or $2k/m_1>b$, $P_{b,k}$ has exactly one positive root.
\end{proof}

\section{Sectional Curvature Pinching Theorems}
In this section, we prove Theorems \ref{thm_sec} and \ref{thm_sec_int}. Firstly we show the following integral formula, which extends Proposition 3.2 of \cite{Shu07} for $k=n/2$ to general $k\ge 1$.
\begin{prop}\label{pinteq}
	For any $n$-dimensional closed $k$-extremal submanifold in $\mathbb{S}^{n+p}$, there holds the following integral equality
	\begin{align}\label{inteq}
		\begin{aligned}
			&\int_{M}\rho^{2k-2}|\nabla \tilde h|^2+(2k-2)\int_{M}\rho^{2k-2}|\nabla \rho |^2\\
			+&\int_{M}\rho^{2k-2}\sum_{\alpha ,i,j,k,l}h_{ij}^{\alpha }(h_{kl}^{\alpha }R_{lijk}
			+h_{li}^{\alpha}R_{lkjk})+\int_{M}\rho^{2k-2}\sum_{\alpha ,\beta ,i,j,k}
			h_{ij}^{\alpha}h_{ki}^{\beta }R_{\beta \alpha jk}^{\bot}\\
			-&n\int_{M}\rho^{2k-2}\sum_{\alpha ,\beta}H^{\alpha }
			\tr(\tilde{A}_{\alpha }\tilde{A}_{\beta}^2)-n\int_{M}\rho^{2k-2}
			\sum_{\alpha, \beta }\tilde {\sigma }_{\alpha \beta }H^{\alpha }H^{\beta }\\
			+&(\frac{n^2}{2k}-n )\int_{M}\rho^{2k}H^2=0.
		\end{aligned}
	\end{align}
\end{prop}
\begin{proof}
	We have by \eqref{selfadjoint}
	\begin{align}\label{adjoint}
		\int_{M}\rho^{2k-2}\sum_{\alpha }\square^{\alpha }(nH^\alpha )-\int_{M}\sum_{\alpha }(nH^\alpha )\square^\alpha \rho^{2k-2}=0.
	\end{align}
	By the definition of $\square^\alpha$, we have
	\begin{equation}\label{boxn}
		\begin{aligned}
			\sum_{\alpha }\square^{\alpha }(nH^\alpha )
			&=n^2\sum_{\alpha,i}H^\alpha H^\alpha_{ii}-\sum_{i,j,k,\alpha}h_{ij}^{\alpha} h_{kkij}^{\alpha}\\
			&=n^2\big(\frac{1}{2}\Delta H^2-|\nabla^{\perp}\mathbf H|^2\big)
			-\sum_{i,j,k,\alpha}h_{ij}^{\alpha} h_{ijkk}^{\alpha}\\
			&\quad+\sum_{\alpha ,i,j,k,l}h_{ij}^{\alpha }(h_{kl}^{\alpha }R_{lijk}+h_{li}^{\alpha }R_{lkjk})
			+\sum_{\alpha ,\beta ,i,j,k} R_{\beta \alpha jk}^{\bot}h_{ij}^{\alpha }h_{ik}^\beta\\
			&=n^2\big(\frac{1}{2}\Delta H^2-|\nabla^{\perp}\mathbf H|^2\big)
			-\frac{1}{2}\Delta S+|\nabla h|^2\\
			&\quad+\sum_{\alpha ,i,j,k,l}h_{ij}^{\alpha }(h_{kl}^{\alpha }R_{lijk}+h_{li}^{\alpha }R_{lkjk})
			+\sum_{\alpha ,\beta ,i,j,k} R_{\beta \alpha jk}^{\bot}h_{ij}^{\alpha }h_{ik}^\beta\\
			&=|\nabla h|^2-n^2|\nabla^{\perp}\mathbf H|^2-\frac{1}{2}\Delta \rho^2+\frac{n(n-1)}{2}\Delta H^2\\
			&\quad+\sum_{\alpha ,i,j,k,l}h_{ij}^{\alpha }(h_{kl}^{\alpha }R_{lijk}+h_{li}^{\alpha }R_{lkjk})
			+\sum_{\alpha ,\beta ,i,j,k} R_{\beta \alpha jk}^{\bot}h_{ij}^{\alpha }h_{ik}^\beta,
		\end{aligned}
	\end{equation}
	where we used the Ricci identity \eqref{eq-Ric-id} in the second equality.

	Multiplying \eqref{boxn} by $\rho^{2k-2}$ and taking integration over $M$, we obtain
	\begin{equation}\label{intnh}
		\begin{aligned}
			\int_{M}\rho^{2k-2}\sum_{\alpha }\square^{\alpha }(nH^\alpha )
			=&
			\int_{M}\rho^{2k-2}(|\nabla h|^2-n^2|\nabla^{\perp}\mathbf {H}|^2)+\frac{n(n-1)}{2}\int_{M}\rho^{2k-2}\Delta H^2-\frac{1}{2}\int_{M}\rho^{2k-2}\Delta \rho^2\\
			&+\int_{M}\rho^{2k-2}\sum_{\alpha ,i,j,k,l}h_{ij}^\alpha(h_{kl}^{\alpha }R_{lijk}+h_{li}^{\alpha }R_{lkjk})+\int_{M}\rho^{2k-2}\sum_{\alpha ,\beta ,i,j,k}h_{ij}^\alpha h_{ki}^\beta R_{\beta \alpha jk}^{\bot}.
		\end{aligned}
	\end{equation}

	From The Euler-Lagrangian equation \eqref{kextremaleq}, we have
	\begin{align}\label{intrhon}
		\begin{aligned}
			-\int_{M}\sum_{\alpha }nH^\alpha \square^{\alpha }\rho^{2k-2}&=-n\int_{M}\rho^{2k-2}\big[\sum_{\alpha,\beta}H^\alpha\tr(A_\alpha A_\beta^2)-
			\sum_{\alpha,\beta}H^\alpha H^{\beta }\sigma_{\alpha\beta }-\frac{n}{2k}\rho^2H^2 \big]\\
			&\quad -n\int_M (n-1)H^2\Delta (\rho^{2k-2})-2n(n-1)\int_{M}\sum_{i,\alpha}(\rho^{2k-2})_{i}H_{i}^{\alpha }H^{\alpha }\\
			&\quad -n(n-1)\int_{M}\rho^{2k-2}\sum_{i,\alpha}H^{\alpha}H^{\alpha}_{ii}\\
			&=-n\int_{M}\rho^{2k-2}\big[\sum_{\alpha,\beta}H^\alpha\tr(\tilde{A}_{\alpha}\tilde{A}_{\beta}^2)+\sum_{\alpha ,\beta }H^\alpha H^\beta \tilde \sigma_{\alpha \beta}+(1-\frac{n}{2k})\rho^2H^2\big]\\
			&\quad -n(n-1)\int_M H^2 \Delta \rho^{2k-2} -n(n-1)\int_{M}\sum_{i}(\rho^{2k-2})_{i}(H^2)_{i}\\
			&\quad -n(n-1)\int_{M}\rho^{2k-2}\big(\frac{1}{2}\Delta H^2-|\nabla^{\bot} \mathbf{H}|^2\big),
		\end{aligned}
	\end{align}
	where we used \eqref{sigma} and \eqref{threeitems} and Stokes' formula.

	Noting that
	\begin{gather*}
		|\nabla \tilde h|^2=\sum_{\alpha,i,j,k}(h_{ijk}^{\alpha}-H_{k}^\alpha \delta_{ij})^2=|\nabla h|^2-n|\nabla^{\perp}\mathbf H|^2,\\
		-\frac{1}{2}\int_{M}\rho^{2k-2}\Delta \rho^2=(2k-2)\int \rho^{2k-2}|\nabla \rho|^2, \quad -\int_M H^2\Delta\rho^{2k-2} =\int_M\sum_{i}(\rho^{2k-2})_i (H^2)_i,
	\end{gather*}
	we can complete the proof by putting \eqref{intnh} and \eqref{intrhon} into \eqref{adjoint}.
\end{proof}

\begin{proof}[Proof of Theorem \ref{thm_sec}]
	By the Ricci equation \eqref{eqR} we have
	\begin{align}\label{norm}
		\begin{aligned}
			\sum_{\alpha ,\beta ,i,j,k}h_{ij}^{\alpha }h_{ki}^{\beta }R_{\beta \alpha jk}^{\bot}
			&=\sum_{\alpha, \beta ,i,j,k,l}h_{ij}^{\alpha }h_{ki}^{\beta }
			(h_{jl}^{\beta }h_{lk}^\alpha -h_{kl}^{\beta}h_{lj}^\alpha )\\
			&=\tr(A_\alpha A_\beta)^2-\tr(A_\alpha^2 A_\beta^2)\\
			&=-\frac{1}{2}\sum_{\alpha ,\beta } N([A_{\alpha },A_{\beta }])\\
			&=-\frac{1}{2}\sum_{\alpha ,\beta } N([\tilde A_{\alpha },\tilde A_{\beta }]).
		\end{aligned}
	\end{align}

	By using of \eqref{eqG}, \eqref{sigma-1}, \eqref{sigma} and \eqref{threeitems}, a direct computation gives
	\begin{align}\label{Iitem}
		\begin{aligned}
			&\phantom{=}\sum_{\alpha ,i,j,k,l}h_{ij}^{\alpha }(h_{kl}^{\alpha }R_{lijk}+h_{li}^{\alpha }R_{lkjk}) \\
			& =n(1+H^2)\rho^2-\sum_{\alpha,\beta}\tilde \sigma_{\alpha \beta }^2
			+n\sum_{\alpha,\beta}H^{\alpha}\tr(\tilde{A}_{\alpha} \tilde{A}_{\beta}^2)
			-\frac{1}{2}\sum_{\alpha,\beta} N([\tilde A_{\alpha},\tilde A_{\beta}]).
		\end{aligned}
	\end{align}

	Now we use the trick of Yau \cite{Yau75}. Inserting \eqref{norm} and \eqref{Iitem} into \eqref{inteq}, we have
	\begin{align}\label{parameter}
		\begin{aligned}
			0&=\int_{M}\rho^{2k-2}|\nabla \tilde h|^2+(2k-2)\int_{M}\rho^{2k-2}|\nabla \rho |^2\\
			&\quad +(a+1)\int_{M}\rho^{2k-2}\sum_{\alpha ,i,j,k,l}h_{ij}^{\alpha }(h_{kl}^{\alpha }R_{lijk}
			+h_{li}^{\alpha}R_{lkjk})+\int_{M}\rho^{2k-2}\sum_{\alpha ,\beta ,i,j,k}
			h_{ij}^{\alpha}h_{ki}^{\beta }R_{\beta \alpha jk}^{\bot}\\
			&\quad -n\int_{M}\rho^{2k-2}\sum_{\alpha ,\beta}H^{\alpha }
			\tr(\tilde{A}_{\alpha }\tilde{A}_{\beta}^2)-n\int_{M}\rho^{2k-2}
			\sum_{\alpha, \beta }\tilde {\sigma }_{\alpha \beta }H^{\alpha }H^{\beta }\\
			&\quad +(\frac{n^2}{2k}-n )\int_{M}\rho^{2k}H^2-a\int_{M}\rho^{2k-2}\sum_{\alpha ,i,j,k,l}h_{ij}^{\alpha }(h_{kl}^{\alpha }R_{lijk}
			+h_{li}^{\alpha}R_{lkjk})\\
			&= \int_{M}\rho^{2k-2}|\nabla \tilde h|^2+(2k-2)\int_{M}\rho^{2k-2}|\nabla \rho |^2
			+a\int_{M}\rho^{2k-2}\sum_{\alpha ,\beta }\tilde \sigma_{\alpha \beta }^2 \\
			&\quad +(a+1)\int_{M}\sum_{\alpha ,i,j,k,l}h_{ij}^{\alpha }(h_{kl}^{\alpha }R_{lijk}+h_{li}^{\alpha }R_{lkjk})
			-\frac{1-a}{2} \int_{M}\rho^{2k-2}\sum_{\alpha ,\beta } N([\tilde A_{\alpha },\tilde A_{\beta  }])\\
			&\quad -(a+1)n\int_{M}\rho^{2k-2}\sum_{\alpha,\beta}H^\alpha\tr(\tilde{A}_{\alpha }\tilde{A}_{\beta}^2)-n\int_{M}\rho^{2k-2}\sum_{\alpha ,\beta }
			\tilde \sigma_{\alpha \beta }H^\alpha H^\beta
			\\
			&\quad +(\frac{n^2}{2k}-n)\int_{M} H^2\rho^{2k}-an\int_{M}H^2\rho^{2k}-an\int_{M}\rho^{2k},
		\end{aligned}
	\end{align}
	where $a\in [0,1]$ is a constant which will be chosen later.

	For a fixed $\alpha$, $n+1\le \alpha \le n+p,$ we can take a local orthonormal frame field $\left \{ e_{1},\cdots ,e_{n} \right \}$ such that $h_{ij}^{\alpha}=\lambda_{i}^\alpha \delta_{ij}$. Then $\tilde h_{ij}^{\alpha}=\mu_{i}^\alpha \delta_{ij}$ with $\mu_{i}^{\alpha}=\lambda_{i}^{\alpha}-H^{\alpha}$ and $\sum_{i}\mu_{i}^{\alpha}=0$.
	We have
	\begin{align*}
		\sum_{i,j,k,l}h_{ij}^{\alpha }(h_{kl}^{\alpha }R_{lijk}+h_{li}^{\alpha }R_{lkjk})
		&=\sum_{i,k}\big((\lambda_i^{\alpha})^2-(\lambda_i^{\alpha}\lambda_k^{\alpha})\big)R_{ikik}\\
		& =\frac{1}{2}\sum_{i,k} (\lambda_{i}^{\alpha}-\lambda_{k}^{\alpha})^2R_{ikik} =\frac{1}{2}\sum_{i,k} (\mu_{i}^{\alpha}-\mu_{k}^{\alpha})^2R_{ikik}\\
		&\ge \frac{1}{2}\sum_{i,k} (\mu_{i}^{\alpha}-\mu_{k}^{\alpha})^2 K_{\min}\\
		&=(\sum_{i,k} \mu_{i}^{\alpha})^2K_{\min}-(\sum_{i}\mu_{i}^{\alpha})(\sum_{k} \mu_{k}^{\alpha}) K_{\min}\\
		&=nN(\tilde{A}_{\alpha})K_{\min}.
	\end{align*}
	Hence, we have
	\begin{align}\label{ineqofsect}
		\sum_{i,j,k,l,\alpha}h_{ij}^{\alpha }(h_{kl}^{\alpha }R_{lijk}+h_{li}^{\alpha }R_{lkjk})\ge nK_{\min}\sum_{\alpha} N(\tilde{A}_{\alpha})=n\rho^2 K_{\min}.
	\end{align}
	Moreover, equality in \eqref{ineqofsect} holds if and only if $R_{ijij}=K_{\min}$ for any $i\ne j$.

	By using \eqref{ineqs} and \eqref{ineq}, we have
	\begin{align}
		\Big|\tr(\tilde{A}_{\beta}\tilde{A}_{\alpha}^2)\Big|=\Big|\sum_{i}\tilde{h}_{ii}^{\beta}(\mu_i^{\alpha})^2\Big|&\le\frac{n-2}{\sqrt{n(n-1)}}\Big(\sum_i(\mu_i^{\alpha})^2\Big)\sqrt{\sum_i(\tilde{h}_{ii}^{\beta})^2}
		\\
		&\le \frac{n-2}{\sqrt{n(n-1)} } N(\tilde{A}_\alpha)\sqrt{N(\tilde{A}_\beta)},
	\end{align}
	hence, by Cauchy-Schwarz inequality we obtain
	\begin{equation}\label{4}
		\sum_{\alpha,\beta}\Big|H^\beta\tr(\tilde{A}_{\beta}\tilde{A}_{\alpha}^2)\Big|
		\le \sum_{\alpha,\beta} \frac{n-2}{\sqrt{n(n-1)} } N(\tilde{A}_\alpha)\sqrt{N(\tilde{A}_\beta)}|H^\beta|
		\le  \frac{n-2}{\sqrt{n(n-1)} } H\rho^3.
	\end{equation}

	We choose a suitable  $e_{n+1},\cdots,e_{n+p}$ such that the $(p\times p)$ matrix $(\tilde \sigma_{\alpha\beta})$ is diagonized, i.e., $\tilde \sigma_{\alpha\beta}=\tilde\sigma_{\alpha\alpha}\delta_{\alpha\beta}.$
	A direct calculation shows that
	\begin{align}\label{H2rho2}
		\sum_{\alpha ,\beta }H^{\alpha }H^{\beta}\tilde \sigma_{\alpha \beta }
		=\sum_{\alpha }(H^\alpha )^2\tilde \sigma_{\alpha \alpha}\le
		\sum_{\alpha }(H^\alpha )^2\sum_{\beta }\tilde \sigma_{\beta \beta}
		=H^2\rho^2.
	\end{align}
	The Cauchy-Schwarz inequality gives
	\begin{equation}\label{5}
		\sum_{\alpha ,\beta }\tilde \sigma_{\alpha \beta }^2=\sum_{\alpha }\tilde \sigma_{\alpha \alpha }^2\ge \frac{1}{p}(\sum_{\alpha }\sigma_{\alpha \alpha })^2=\frac{1}{p}\rho^4.
	\end{equation}

	The DDVV inequality \eqref{DDVVineq} implies
	\begin{align}\label{3}
		\sum_{\alpha, \beta }N([\tilde A_{\alpha },\tilde A_{\beta  }])\le \sgn(p-1)\rho^4.
	\end{align}

	By using \eqref{ineqofsect}, \eqref{4}, \eqref{H2rho2}, \eqref{5} and \eqref{3}, it follows from \eqref{parameter} that
	\begin{align}\label{afterpara}
		\begin{aligned}
			0\ge &\int_{M}\rho^{2k-2}|\nabla \tilde h|^2
			+(2k-2)\int_{M}\rho^{2k-2}|\nabla \rho |^2+\frac{a}{p}\int_{M}\rho^{2k+2}\\
			&+(a+1)n\int_{M}\rho^{2k}K_{\min}-\frac{1-a}{2}\sgn(p-1)\int_{M}\rho^{2k+2}\\
			&-(a+1)\frac{n(n-2)}{\sqrt{n(n-1)}}\int_{M}\rho^{2k+1}H-(a+1)n\int_{M}\rho^{2k}H^2\\
			&+(\frac{n^2}{2k}-n )\int_{M}H^2\rho^{2k}-an\int_{M}\rho^{2k}.
		\end{aligned}
	\end{align}
	Now taking $a=\frac{p}{p+2}\sgn(p-1)$, we have
	\begin{equation}\label{pointk}
		\begin{aligned}
			0\ge &\int_{M}\rho^{2k-2}|\nabla \tilde h|^2
				+(2k-2)\int_{M}\rho^{2k-2}|\nabla \rho |^2\\
				&+(a+1)n\int_{M}\rho^{2k}\Big(K_{\min}-C_1(n,p,H,\rho,k)\Big).
		\end{aligned}
	\end{equation}

	Since $k\ge 1$, from the pinching condition \eqref{cond-sec} we obtain
	\begin{equation}\label{pointk-1}
		\begin{aligned}
			0\ge &\int_{M}\rho^{2k-2}|\nabla \tilde h|^2
				+(2k-2)\int_{M}\rho^{2k-2}|\nabla \rho |^2\\
				&+(a+1)n\int_{M}\rho^{2k}\Big(K_{\min}-C_1(n,p,H,\rho,k)\Big)\ge 0.
		\end{aligned}
	\end{equation}
	This implies $|\nabla \tilde h|^2\equiv 0$ on $M$ and all the inequalities in \eqref{4}, \eqref{H2rho2}, \eqref{3} become equalities.

	From $|\nabla \tilde h|^2=\sum_{\alpha,i,j,k}(h_{ijk}^{\alpha}-H_{k}^\alpha \delta_{ij})^2=0$ we have $h_{ijk}^{\alpha}=H_{k}^\alpha \delta_{ij}$, which implies $h_{ijk}^{\alpha}=0, H_{k}^\alpha=0$ for all $i,j,k,\alpha$, i.e., both the second fundamental form and the mean curvature of $M$ are parallel. Moreover, both $H$ and $\rho^2=S-nH^2$ are constant.

	If $\rho^2=0$, then $M$ is totally umbilical. Next, we discuss the case of $\rho^2>0$.

	\textbf{Case (i):} $p\ge 2$. Equalities \eqref{3} and \eqref{H2rho2} imply $p=2$ (cf. Lemma \ref{lem-DDVV}) and $H^2\rho^2=0$.
	Therefore, $M$ is minimal and $\sec^M\ge \frac{p}{2(p+1)}$. A rigidity theorem for minimal submanifolds due to Gu-Xu (cf. \cite[Theorem 1]{GX12}) says that  $M$ is the Veronese surface in $\mathbb{S}^4$ with $n=2$.

	\textbf{Case (ii):} $p=1$. The parameter $a=0$ and equalities in \eqref{H2rho2}, \eqref{5}, and \eqref{3} hold automatically. We conclude that $M=\mathbb{S}^{m}(a)\times \mathbb{S}^{n-m}(\sqrt{1-a^2})$ by Lawson's classification \cite{Law69}. From \eqref{pointk-1} we have
	\begin{equation}\label{eq-Kmin}
		0=K_{\min}=\frac{n-2}{\sqrt{n(n-1)}}H\rho+\big(2-\frac{n}{2k}\big)H^2.
	\end{equation}

	When $n=2$, since $k\ge 1$, we conclude $H=0$, and $M$ is the Clifford torus $C_{1,1}$ (cf. Theorem \ref{thm-iso-2}).

	When $n\ge 3$, if $H=0$, then $M=C_{m,n-m}$.
	We must have $n=2m$ and $M=C_{m,m}$ by Theorem \ref{thm-iso-2}.

	If $H\neq 0$, \eqref{4} can be proven by Lemma \ref{ineqs-hypersurface}, and the equality case implies $m=1$ without of loss of generality. From \eqref{eq-Kmin} we have $\frac{n-2}{\sqrt{n(n-1)}}\rho=(\frac{n}{2k}-2)H>0$, then $k<n/4$.  But $m=1$ doesn't satisfy $2k<m<n-2k$. Hence, this case cannot occur by Theorem \ref{thm-iso-2}.
\end{proof}

\begin{proof}[Proof of Theorem \ref{thm_sec-n}]
	We modified \eqref{afterpara} by using (cf. Lemma \ref{lem-Itoh})
	\begin{align}\label{3'}
		\sum_{\alpha, \beta }N([\tilde A_{\alpha },\tilde A_{\beta  }])\le n\sum_{\alpha, \beta }\tilde{\sigma}_{\alpha\beta}^2
	\end{align}
	instead of \eqref{5} and \eqref{3}. Then we have
	\begin{align}\label{afterpara'}
		\begin{aligned}
			0\ge &\int_{M}\rho^{2k-2}|\nabla \tilde h|^2
			+(2k-2)\int_{M}\rho^{2k-2}|\nabla \rho |^2\\
			&+(a+1)n\int_{M}\rho^{2k}K_{\min}+\Big(\frac{a}{n}-\frac{1-a}{2}\Big)\int_{M}\rho^{2k-2}\sum_{\alpha,\beta}N([\tilde{A}_\alpha,\tilde{A}_\beta])\\
			&-(a+1)\frac{n(n-2)}{\sqrt{n(n-1)}}\int_{M}\rho^{2k+1}H-(a+1)n\int_{M}\rho^{2k}H^2\\
			&+(\frac{n^2}{2k}-n )\int_{M}H^2\rho^{2k}-an\int_{M}\rho^{2k}.
		\end{aligned}
	\end{align}
	By taking $a=\frac{n}{n+2}$, we obtain
	\begin{equation}\label{pointk'}
		\begin{aligned}
			0\ge &\int_{M}\rho^{2k-2}|\nabla \tilde h|^2
				+(2k-2)\int_{M}\rho^{2k-2}|\nabla \rho |^2\\
				&+(a+1)n\int_{M}\rho^{2k}\Big(K_{\min}-C_1'(n,H,\rho,k)\Big)\ge 0.
		\end{aligned}
	\end{equation}
	The rest of the proof is similar with the proof of Theorem \ref{thm_sec}, just by using Itoh's classification \cite{Ito75} for minimal submanifolds when $p\ge 2$ and excluding the case of $p=1$.
\end{proof}

\begin{proof}[Proof of Theorem \ref{thm_sec_int}]
	We start from \eqref{pointk}.
	\begin{align}\label{intnabla}
		\begin{aligned}
			0\ge &\int_{M}\rho^{2k-2}|\nabla \tilde h|^2
				+(2k-2)\int_{M}\rho^{2k-2}|\nabla \rho|^2\\
				&+(a+1)n\int_{M}\rho^{2k}\Big(K_{\min}-C_1(n,p,H,\rho,k)\Big)
		\end{aligned}
	\end{align}
	with $a=\frac{p}{p+2}\sgn(p-1)$.

	By using \eqref{ineqrho} and letting $\epsilon\to 0$, we have
	\begin{equation}\label{eq-sec-int-part1}
		\int_{M}\rho^{2k-2}|\nabla \tilde h|^2
				+(2k-2)\int_{M}\rho^{2k-2}|\nabla \rho|^2\ge \frac{2kn-n+2}{nk^2}\left\{\mathfrak{A}(n,t)\|\rho^{2k}\|_{\frac{n}{n-2} }
				-\mathfrak{B}(n,t)\int_{M}(1+H^2) \rho^{2k}\right\}.
	\end{equation}

	Since $\delta=\Phi-C_1>0$, we have $\besec-\delta\ge C_1-K_{\min}$ by the definition of $\besec$. Hence, by using the H\"{o}lder inequality we have
	\begin{equation}\label{eq-sec-int-part2}
		\int_{M}\rho^{2k}\Big(K_{\min}-C_1(n,p,H,\rho,k)\Big)\ge \int_{M}\delta\rho^{2k}-\|\rho^{2k}\|_{\frac{n}{n-2}} \|\besec \|_{\frac{n}{2}}.
	\end{equation}
	Since $M$ is compact, there exist constants $\delta_0>0$ and $H_0\ge 0$ such that $\int_{M}\delta\rho^{2k}=\delta_0\int_{M}\rho^{2k}$ and $\int_{M}(1+H^2)\rho^{2k}=(1+H_0^2)\int_{M}\rho^{2k}$.
	Putting \eqref{eq-sec-int-part1} and \eqref{eq-sec-int-part2} into \eqref{intnabla} and choosing $t_1$ such that $\frac{2kn-n+2}{nk^2} \mathfrak{B}(n,t_1) (1+H_0^2)=(a+1)n\delta_0$, we obtain
	\begin{equation*}
		0\ge \|\rho^{2k}\|_{\frac{n}{n-2}}\Big(\frac{2kn-n+2}{nk^2}\mathfrak{A}(n,t_1)-(a+1)n\|\besec\|_{\frac{n}{2}}\Big).
	\end{equation*}
	Hence,  $M$ is totally umbilical provided $\|\besec\|_{\frac{n}{2}}<\epsilon_1$ by setting $\epsilon_1=\frac{1}{(a+1)n}\frac{2kn-n+2}{nk^2}\mathfrak{A}(n,t_1)$.
\end{proof}

The proof of Theorem \ref{thm_sec_int-n} is analogous to the proof of Theorem \ref{thm_sec_int}, so we omit it.

\section{Ricci Curvature Pinching Theorems}
In this section, we prove Theorems \ref{thm_ricci} and \ref{thm_ricci_int}.

\begin{proof}[Proof of Theorem \ref{thm_ricci}]
	Putting \eqref{norm} and \eqref{Iitem} into \eqref{inteq}, and using \eqref{H2rho2}, we have
	\begin{equation}\label{eq-Ricci}
		\begin{aligned}
			0=& \int_{M}\rho^{2k-2}|\nabla \tilde{h}|^2+(2k-2)\int_{M}\rho^{2k-2}|\nabla \rho|^2\\
			&+n\int_{M}\rho^{2k}-\int_{M}\rho^{2k-2}\sum_{\alpha,\beta}\tilde \sigma_{\alpha \beta}^2
			-\int_{M}\rho^{2k-2}\sum_{\alpha,\beta} N([\tilde A_{\alpha},\tilde A_{\beta}])\\
			&+n\int_{M}\rho^{2k}H^2-n\int_{M}\rho^{2k-2}
			\sum_{\alpha, \beta}\tilde{\sigma}_{\alpha \beta}H^{\alpha}H^{\beta}+(\frac{n^2}{2k}-n)\int_{M}\rho^{2k}H^2\\
			\ge & \int_{M}\rho^{2k-2}|\nabla \tilde{h}|^2+(2k-2)\int_{M}\rho^{2k-2}|\nabla \rho|^2+(\frac{n^2}{2k}-n)\int_{M}\rho^{2k}H^2\\
			&+\int_{M}\rho^{2k-2}\Big(n\rho^{2}-\sum_{\alpha,\beta}\tilde \sigma_{\alpha \beta}^2
			-\sum_{\alpha,\beta} N([\tilde A_{\alpha},\tilde A_{\beta}])\Big).
		\end{aligned}
	\end{equation}
	Next we use Ejiri's trick \cite{Eji79} to estimate the last term of \eqref{eq-Ricci}.

	By the Gauss equation \eqref{eqG}, we have
	\begin{equation}\label{eqG-tr-free}
		R_{ii}=(n-1)+(n-2)\sum_{\alpha}H^{\alpha}\tilde h_{ii}^{\alpha}+(n-1)H^2
		-\sum_{\alpha ,j}(\tilde h_{ij}^\alpha)^2,
	\end{equation}
	For a fixed $\alpha$, $n+1\le \alpha \le n+p,$ we can take a local orthonormal frame field $\left \{ e_{1},\cdots ,e_{n} \right \}$ such that $\tilde h_{ij}^{\alpha}=\mu_{i}^\alpha \delta_{ij}$. Then \eqref{eqG-tr-free} gives
	\begin{align}
		\sum_j\sum_{\beta \neq \alpha}(\tilde h_{ij}^{\beta})^2\le (n-1)(1+H^2)+
		(n-2)\sum_{\gamma}H^\gamma \tilde h_{ii}^\gamma -(\mu_{i}^\alpha)^2-\Ric_{\min} \mbox{ for each $i$},
	\end{align}
	and then
	\begin{align}
		\begin{aligned}
			\sum_{\beta}N([\tilde A_{\alpha},\tilde A_{\beta}])&=\sum_{\beta\neq \alpha} N([\tilde A_{\alpha},\tilde A_{\beta}])=\sum_{\beta \ne \alpha}\sum_{i,j}
			(\tilde{h}_{ij}^\beta )^2(\mu_{i}^{\alpha}-\mu_{j}^{\alpha})^2\\
			&\le 2\sum_{i,j}\sum_{\beta \neq \alpha}(\tilde{h}_{ij}^\beta)^2\Big((\mu_{i}^{\alpha})^2+(\mu_{j}^{\alpha})^2\Big)= 4\sum_{i,j}\sum_{\beta \neq \alpha}(\tilde{h}_{ij}^\beta )^2(\mu_{i}^{\alpha})^2\\
			&\le4\sum_{i}\Big((n-1)(1+H^2)+
			(n-2)\sum_{\gamma}H^\gamma \tilde h_{ii}^\gamma -(\mu_{i}^\alpha)^2-\Ric_{\min}\Big)(\mu_{i}^\alpha)^2\\
			&=4\Big((n-1)(1+H^2)-\Ric_{\min}\Big)N(\tilde{A}_\alpha)+4
			(n-2)\sum_{\gamma}H^\gamma \tr(\tilde{A}_\gamma\tilde{A}_\alpha^2)-4N(\tilde{A}_\alpha^2).
		\end{aligned}
	\end{align}
	Making summation over $\alpha$, we have
	\begin{equation}
			\sum_{\alpha ,\beta}N([\tilde A_{\alpha},\tilde A_{\beta}])\le
			4\Big((n-1)(1+H^2)-\Ric_{\min}\Big)\rho^2+
			\frac{4(n-2)^2}{\sqrt{n(n-1)} } H\rho^3-\frac{4}{n}\sum_{\alpha}\big(N(\tilde{A}_\alpha)\big)^2,
	\end{equation}
	where we used \eqref{4} and the Cauchy-Schwarz inequality.
	We denote
	\begin{equation*}
		\tilde{C}_2=(n-2)+\frac{(n-2)^2}{\sqrt{n(n-1)} } H\rho+(n-1)H^2,
	\end{equation*}
	then \eqref{scalar} gives
	\begin{equation}\label{eq-Gauss-bound}
		\Ric_{\min}\le \frac{R}{n}\le (n-1)(1+H^2)-\frac{\rho^2}{n}\le \frac{n-\rho^2}{n}+\tilde{C}_2.
	\end{equation}
	Noting \eqref{5}, we obtain
	\begin{align}\label{eq-Q}
		\begin{aligned}
			&\quad\,\, n\rho^{2}-\sum_{\alpha,\beta}\tilde \sigma_{\alpha \beta}^2
			-\sum_{\alpha,\beta} N([\tilde A_{\alpha},\tilde A_{\beta}])\\
			&\ge n\rho^2-4\Big((n-1)(1+H^2)+\frac{(n-2)^2}{\sqrt{n(n-1)} } H\rho-\Ric_{\min}\Big)\rho^2+\frac{4-n}{n}\sum_{\alpha}\big(N(\tilde{A}_\alpha)\big)^2\\
			&\ge -4\Big((n-2)+(n-1)H^2+\frac{(n-2)^2}{\sqrt{n(n-1)} } H\rho-\Ric_{\min}\Big)\rho^2+\frac{n-4}{n}\rho^2(n-\rho^2)\\
			&\ge -4(\tilde{C}_2-\Ric_{\min})\rho^2+(n-4)(\Ric_{\min}-\tilde{C}_2)\rho^2\\
			&= n\rho^2(\Ric_{\min}-\tilde{C}_2).
		\end{aligned}
	\end{align}
	Hence, putting \eqref{eq-Q} into \eqref{eq-Ricci} we obtain
	\begin{equation}\label{eq-Ricci-key}
		\begin{aligned}
			0\ge& \int_{M}\rho^{2k-2}|\nabla \tilde{h}|^2+(2k-2)\int_{M}\rho^{2k-2}|\nabla \rho|^2+\int_{M}\rho^{2k}n\Big(\Ric_{\min}-\tilde{C}_2+(\frac{n}{2k}-1)H^2\Big)\\
			=& \int_{M}\rho^{2k-2}|\nabla \tilde{h}|^2+(2k-2)\int_{M}\rho^{2k-2}|\nabla \rho|^2+n\int_{M}\rho^{2k}(\Ric_{\min}-C_2).
		\end{aligned}
	\end{equation}
	Since $k\ge 1$, if $\Ric_{\min}\ge C_2$, then
	\begin{equation}\label{eq-Ricci-start}
			0\ge \int_{M}\rho^{2k-2}|\nabla \tilde{h}|^2+(2k-2)\int_{M}\rho^{2k-2}|\nabla \rho|^2+n\int_{M}\rho^{2k}(\Ric_{\min}-C_2)\ge 0.
	\end{equation}
	As pointed in the proof of Theorem \ref{thm_sec}, we have $|\nabla h|^2=|\nabla^{\bot} \mathbf{H}|^2=0$, and both $H$ and $\rho^2$ are constant.

	If $\rho^2=0$, then $M$ is totally umbilical. Next, we discuss the case of $\rho^2>0$.
	From \eqref{eq-Ricci-start} we have
	\begin{equation}
		\Ric_{\min}= C_2.
	\end{equation}

	When $n\ge 5$, all the inequalities in \eqref{eq-Gauss-bound} become equalities, which implies $H=0$.
	Then $M$ is $C_{m,m}$ with $n=2m$ by a rigidity theorem of Ejiri for minimal submanifolds \cite{Eji79}.

	When $n=4$, since $k\ge 1$ we have
	\begin{equation}\label{eq-Ric-pmc-cond}
		\Ric^M\ge C_2\ge (n-2)+(n-\frac{n}{2k})H^2\ge (n-2)(1+H^2).
	\end{equation}
	By a rigidity result of Xu-Gu \cite[Theorem 3.3]{XG13} for submanifolds with parallel mean curvature,
	$M$ is either $\mathbb{S}^2(\frac{1}{\sqrt{2(1+H^2)}})\times \mathbb{S}^2(\frac{1}{\sqrt{2(1+H^2)}})$ in $\mathbb{S}^5(\frac{1}{\sqrt{1+H^2}})$,
	or the complex projective space $\mathbb{C}P^2(\frac{4}{3}(1+H^2))$ with constant holomorphic sectional curvature $\frac{4}{3}(1+H^2)$ in $\mathbb{S}^7(\frac{1}{\sqrt{1+H^2}})$.
	These two candidates are Einstein with $\Ric^M=(n-2)(1+H^2)$. Hence, we must have $H=0$ by noticing \eqref{eq-Ric-pmc-cond}.
\end{proof}

\begin{proof}[Proof of Theorem \ref{thm_ricci_int}]
	Since $\delta=(n-1)\Phi-C_2>0$, we have $\beric-\delta\ge C_2-\Ric_{\min}$ by the definition of $\beric$. From \eqref{eq-Ricci-key} (cf. \eqref{eq-sec-int-part1} and \eqref{eq-sec-int-part2}) we have
	\begin{equation}
		0\ge  \frac{2kn-n+2}{nk^2}\mathfrak{A}(n,t)\|\rho^{2k}\|_{\frac{n}{n-2}}-n\|\rho^{2k}\|_{\frac{n}{n-2}} \|\beric \|_{\frac{n}{2}}
	\end{equation}
	by choosing $\delta_0, H_0, t_2$ such that
	\begin{equation*}
		\int_{M}\delta\rho^{2k}=\delta_0\int_{M}\rho^{2k}, \int_{M}(1+H^2)\rho^{2k}=(1+H_0^2)\int_{M}\rho^{2k}, \frac{2kn-n+2}{nk^2} \mathfrak{B}(n,t_2) (1+H_0^2)=n\delta_0.
	\end{equation*}
	Hence,  $M$ is totally umbilical provided $\|\beric\|_{\frac{n}{2}}<\epsilon_2$ by setting $\epsilon_2=\frac{1}{n}\frac{2kn-n+2}{nk^2}\mathfrak{A}(n,t_2)$.
\end{proof}

\section{Scalar curvature pinching theorems}
In this section, we prove rigidity theorems under scalar curvature pinching conditions.
\begin{proof}[Proof of Theorem \ref{thm_scal}]
	Putting \eqref{norm} and \eqref{Iitem} into \eqref{inteq} we have
	\begin{align}\label{parameter-scalar}
		\begin{aligned}
			0&=\int_{M}\rho^{2k-2}|\nabla \tilde h|^2+(2k-2)\int_{M}\rho^{2k-2}|\nabla \rho |^2\\
			&\quad +n\int_{M}(1+H^2)\rho^{2k}-\int_{M}\rho^{2k-2}\Big(\sum_{\alpha,\beta}\tilde \sigma_{\alpha \beta }^2
			+\sum_{\alpha,\beta} N([\tilde A_{\alpha},\tilde A_{\beta}])\Big)\\
			&\quad -n\int_{M}\rho^{2k-2}
			\sum_{\alpha, \beta }\tilde {\sigma }_{\alpha \beta }H^{\alpha }H^{\beta } +(\frac{n^2}{2k}-n )\int_{M}\rho^{2k}H^2\\
			&\ge \int_{M}\rho^{2k-2}|\nabla \tilde h|^2+(2k-2)\int_{M}\rho^{2k-2}|\nabla \rho |^2\\
			&\quad +n\int_{M}(1+H^2)\rho^{2k}-\int_{M}\rho^{2k-2}\big(1+\frac{1}{2}\sgn(p-1)\big)\rho^4\\
			&\quad -n\int_{M}\rho^{2k}H^2 +(\frac{n^2}{2k}-n )\int_{M}\rho^{2k}H^2\\
			&\ge \int_{M}\rho^{2k}\Big(n+n(\frac{n}{2k}-1)H^2-\big(1+\frac{1}{2}\sgn(p-1)\big)\rho^2\Big)\ge 0
		\end{aligned}
	\end{align}
	provided
	\begin{equation*}
		\rho^2\le \frac{n+n(\frac{n}{2k}-1)H^2}{1+\frac{1}{2}\sgn(p-1)}=C_3(n,p,H,k).
	\end{equation*}
	Here we used \eqref{LL-ineq} and \eqref{H2rho2}.

	Hence, by the same arguments as in the proof of Theorem \ref{thm_sec}, $M$ has parallel second fundamental form and parallel mean curvature; both $\rho^2$ and $H^2$ are constant.
	We obtain either $\rho^2=0$ or $\rho^2=C_3$.
	When $\rho^2=C_3$, if $p=1$, then $M=T_{m,k}$;
	if $p\ge 2$, then $H^2=0$ (cf. the equality case in \eqref{H2rho2} and Lemma \ref{lem-LL}) and $S\le \frac{2}{3}n$, so $M$ is the Veronese surface in $\mathbb{S}^4$ by \cite[Theorem 3]{LL92}.
\end{proof}

\begin{proof}[Proof of Theorems \ref{thm_scal_int} and Theorems \ref{thm_scal_int'}]
	Denote $\tau(p)=1+\frac{1}{2}\sgn(p-1)$.
	Since $\delta=C_3-\Phi>0$, we have $C_3-\rho^2\ge \delta-(\rho^2-\Phi)_{+}$.
	From \eqref{parameter-scalar} (cf. \eqref{eq-sec-int-part1} and \eqref{eq-sec-int-part2}) we have
	\begin{equation}
		0\ge \frac{2kn-n+2}{nk^2}\mathfrak{A}(n,t)\|\rho^{2k}\|_{\frac{n}{n-2}}-\tau(p)\|\rho^{2k}\|_{\frac{n}{n-2}} \|(\rho^2-\Phi)_{+}\|_{\frac{n}{2}}
	\end{equation}
	by choosing $\delta_0, H_0, t_3$ such that
	\begin{equation*}
		\int_{M}\delta\rho^{2k}=\delta_0\int_{M}\rho^{2k}, \int_{M}(1+H^2)\rho^{2k}=(1+H_0^2)\int_{M}\rho^{2k}, \frac{2kn-n+2}{nk^2} \mathfrak{B}(n,t_3) (1+H_0^2)=\tau(p)\delta_0.
	\end{equation*}
	Hence,  $M$ is totally umbilical provided $\|(\rho^2-\Phi)_{+}\|_{\frac{n}{2}}<\epsilon_3$ by setting $\epsilon_3=\frac{1}{\tau(p)}\frac{2kn-n+2}{nk^2}\mathfrak{A}(n,t_3)$.
	This proves Theorem \ref{thm_scal_int}.

	When $1\le k<n/2$, since $\frac{n^2}{2k}-n>0$, we can control $\mathfrak{B}(n,t)\int_{M}H^2\rho^{2k}$ by $\int_{M}(\frac{n^2}{2k}-n)H^2\rho^{2k}$. From \eqref{parameter-scalar} we have
\begin{equation*}
		0\ge \frac{2kn-n+2}{nk^2}\mathfrak{A}(n,t)\|\rho^{2k}\|_{\frac{n}{n-2}}-\tau(p)\|\rho^{2k}\|_{\frac{n}{n-2}} \|(\rho^2-\Phi)_{+}\|_{\frac{n}{2}}
\end{equation*}
by setting $\delta'=(1-\frac{1}{3}\sgn(p-1))n-\Phi$ and choosing $\delta_0', t_3'$ such that
\begin{equation*}
	\int_{M}\delta'\rho^{2k}=\delta_0'\int_{M}\rho^{2k}, \quad \frac{2kn-n+2}{nk^2}\mathfrak{B}(n,t_3')=\tau(p)\min\Big\{\frac{n^2}{2k}-n, \delta_0'\Big\},
\end{equation*}
Hence,  $M$ is totally umbilical provided $\|(\rho^2-\Phi)_{+}\|_{\frac{n}{2}}<\epsilon_3'$ by setting $\epsilon_3'=\frac{1}{\tau(p)}\frac{2kn-n+2}{nk^2}\mathfrak{A}(n,t_3')$.
This proves Theorem \ref{thm_scal_int'}.
\end{proof}


\end{document}